\documentclass[11pt]{alggeom}
\usepackage{amsmath}
\usepackage{amsthm}
\usepackage{amssymb,amsfonts}
\usepackage{hyperref}
\usepackage{latexsym}
\usepackage{todonotes}
\usepackage{nicefrac}
\usepackage{epsfig}
\usepackage{stmaryrd}
\usepackage{setspace}
\usepackage{enumerate}
\usepackage[all]{xypic}
\usepackage{bbm,ifpdf,tikz}
\ifpdf
\usepackage{pdfsync}
\fi

\newtheorem{theorem}{Theorem}[section]

\newtheorem{lemma}[theorem]{Lemma}
\newtheorem{proposition}[theorem]{Proposition}
\newtheorem{corollary}[theorem]{Corollary}

\newtheorem{conjecture}[theorem]{Conjecture}

{
\theoremstyle{definition}

\newtheorem{example}[theorem]{Example}

\newtheorem{remark}[theorem]{Remark}
}

\newenvironment{proofmonspan}{\noindent {\bf Proof of Theorem \ref{monomial span}:}}{\qed \par}

\newcommand{\excise}[1]{}
\newcommand{\Spec}{\operatorname{Spec}}

\newcommand{\Proj}{\operatorname{Proj}}

\newcommand{\supp}{\operatorname{Supp}}

\renewcommand{\Im}{\operatorname{Im}}

\newcommand{\Ker}{\operatorname{Ker}}
\newcommand{\codim}{\operatorname{codim}}
\renewcommand{\dim}{\operatorname{dim}}
\newcommand{\Ann}{\operatorname{Ann}}
\newcommand{\Sym}{\operatorname{Sym}}
\newcommand{\rk}{\operatorname{rk}}

\renewcommand{\and}{\qquad\text{and}\qquad}

\newcommand{\Z}{\mathbb{Z}}
\newcommand{\Q}{\mathbb{Q}}
\newcommand{\N}{\mathbb{N}}
\newcommand{\R}{\mathbb{R}}
\newcommand{\C}{\mathbb{C}}

\newcommand{\la}{\leftarrow}

\newcommand{\mmod}{/\!\!/}

\renewcommand{\la}{\lambda}

\newcommand{\cs}{{\C^\times}}

\renewcommand{\a}{\alpha}

\renewcommand{\b}{\beta}

\newcommand{\Hom}{\operatorname{Hom}}

\newcommand{\becircled}{\mathaccent "7017}

\newcommand{\bS}{\mathbb{S}}

\newcommand{\nc}{\newcommand}
\nc{\cq}{/\!\!/}
\nc{\red}{{U}}
\nc{\zred}{{U_\fz}}
\nc{\sred}{{\cal U}}
\nc{\szred}{{\cal U}_{\fz}}
\nc{\sEnd}{{\cal E}\mathit{nd}}
\nc{\sHom}{{\cal H}\mathit{om}}
\nc{\Pic}{\operatorname{Pic}}
\nc{\cone}[1]{\text{\textswab C}(#1)}
\nc{\coneZ}[1]{\text{\textswab C}_\Z(#1)}

\newcommand{\fz}{{\mathfrak{z}}}

\nc{\Gd}{G'}

\newcommand{\End}{\operatorname{End}}

\nc{\Ree}{\EuScript{R}}

\nc{\muq}{\boldsymbol{\mu}}

\renewcommand{\th}{\theta}

\newcommand{\ith}{i^\text{th}}

\newcommand{\Lie}{\operatorname{Lie}}

\newcommand{\tX}{\tilde{X}}

\renewcommand{\L}{\Lambda}
\newcommand{\hL}{\hat{\Lambda}}
\newcommand{\QH}{QH^*_\bG(\tX; \C)}
\newcommand{\QHT}{QH^*_{\bT}(\tX; \C)}
\newcommand{\QHpol}{\QH_{\operatorname{pol}}}
\newcommand{\QHTpol}{\QHT_{\operatorname{pol}}}
\newcommand{\QHrat}{\QH_{\operatorname{rat}}}
\newcommand{\QHloc}{\QH_{\operatorname{loc}}}
\newcommand{\IH}{I\! H^*_\bG(X; \C)}
\newcommand{\IHT}{I\! H^*_\bT(X; \C)}
\newcommand{\IHz}{I\! H^*(X; \C)}

\newcommand{\mm}{\mathfrak{m}}
\newcommand{\init}{\operatorname{in}}
\newcommand{\SRind}{\SR_{\operatorname{ind}}}
\newcommand{\SRbc}{SR_{\operatorname{bc}}}
\newcommand{\SR}{SR}
\newcommand{\OT}{OT}
\newcommand{\OTh}{\OT_\hbar}
\newcommand{\Ih}{I_\hbar}
\newcommand{\Hilb}{\operatorname{Hilb}}
\newcommand{\bG}{\mathbf{G}}
\newcommand{\bT}{\mathbf{T}}

\newcommand{\Cbar}{\overline{C}}
\newcommand{\imin}{j_{\operatorname{max}}}

\newcommand{\fch}{f_C}
\newcommand{\fcz}{f_{C,0}}

\begin{document}
\spacing{1.5}
\noindent{\Large\bf Intersection cohomology and quantum cohomology\\
of conical symplectic resolutions}
\bigskip

\spacing{1.2}
\noindent{\bf Michael McBreen}\\
Department of Mathematics, \'Ecole Polytechnique F\'ed\'erale de Lausanne, CH-1015 Lausanne\\

\noindent{\bf Nicholas Proudfoot}\footnote{Supported by NSF grant DMS-0950383.}\\
Department of Mathematics, University of Oregon,
Eugene, OR 97403
\\
{\small
\begin{quote}
\noindent {\em Abstract.}
For any conical symplectic resolution, we give a conjecture relating the intersection cohomology
of the singular cone to the quantum cohomology of its resolution.  We prove this conjecture for
hypertoric varieties, recovering the ring structure on hypertoric intersection cohomology that was
originally constructed by Braden and the second author.
\end{quote} }

\section{Introduction}
Let $\tX$ be a conical symplectic resolution of $X$; examples include the Springer resolution, Hilbert schemes
of points on ALE spaces,
quiver varieties, hypertoric varieties, and transverse slices to Schubert varieties in the affine Grassmannian.
The purpose of this paper is to state a conjectural relationship between the intersection cohomology of $X$
and the quantum cohomology of $\tX$ (Conjecture \ref{main}), and to prove this conjecture for hypertoric varieties
(Theorem \ref{hypertoric main}).

Before describing the conjecture itself, we say a few words about the significance of the two sides.
Intersection cohomology groups of quiver varieties were shown by Nakajima to coincide with multiplicity
spaces of simple modules in standard modules over a specialized quantum loop algebra \cite[3.3.2 \& 14.3.10]{Na}.
The equivariant intersection cohomology of a hypertoric variety is isomorphic to the Orlik-Terao algebra of a hyperplane
arrangement \cite[4.5]{TP08}, which has been the subject of much recent study \cite{Terao-OT,PS,
ST-OT,Schenck-OT,DR-OT,DGS-OT,Dinh-OT,Liu-OT}.  The equivariant intersection cohomology groups of
slices in the affine Grassmannian, with the equivariant parameters specialized to generic values, are isomorphic
via the geometric Satake correspondence to weight spaces of simple representations for the Langlands dual group
\cite[3.11 \& 5.2]{G-GS}.

On the quantum cohomology side, Okounkov and Pandharipande studied the 
Hilbert scheme of points in the plane \cite{OP}, and Maulik and Oblomkov studied more generally
the Hilbert scheme of points on an ALE space of type A \cite{Maulik-Oblomkov}.
Braverman, Maulik, and Okounkov computed the quantum cohomology
of the Springer resolution \cite{BMO} and gave some indication of how to proceed for arbitrary conical symplectic resolutions.
This program was carried out for quiver varieties by Maulik and Okounkov \cite{QGQC}, who relate their quantum cohomology
to the representation theory of the Yangian, and for hypertoric varieties by Shenfeld and the first author \cite{McBS}.
This last paper gives an explicit generators-and-relations presentation of the hypertoric quantum cohomology ring, which is a large
part of what makes the hypertoric case of our conjecture more tractable than the others. 

Our conjecture very roughly says that the intersection cohomology of $X$ is isomorphic to the quantum cohomology
of $\tX$ specialized at $q=1$.  Of course, this cannot quite be correct as stated.  The first problem is that
quantum cohomology is an algebra over power series, not polynomials, so it does not make sense to set $q$ equal to 1.
We address this problem simply by working with the subalgebra of quantum cohomology generated by ordinary cohomology
and polynomials in $q$, which is tautologically an algebra over polynomials in $q$.  The second problem is that we
would expect any specialization of the quantum parameters to have the same dimension as the cohomology of $\tX$,
and the intersection cohomology of $X$ is strictly smaller than that.  Indeed, the actual statement involves
taking a quotient of this specialization by the annihilator of $\hbar$, where $\hbar$ is the equivariant parameter for
the conical action of the multiplicative group.  A precise formulation of the conjecture appears in Section \ref{sec:statement}.
We work out the example of $T^*\mathbb{P}^1$ in explicit detail, and give a heuristic reason why we would expect the conjecture
to hold in general (Example \ref{basic example}).

One of the interesting consequences of our conjecture would be that the intersection cohomology of $X$ inherits
a ring structure from the quantum cohomology of $\tX$.  As mentioned above, the intersection cohomology of
a hypertoric variety is already known to have a natural ring structure by work of Braden and the second author \cite{TP08}.
However, the techniques in that paper were very combinatorial, and it was never adequately explained
why such a structure should exist.  We regard the proof of our conjecture for hypertoric varieties as an explanation
of where this mysterious ring structure comes from.  See Section \ref{sec:BP} for a more detailed discussion of the relationship
between our results and those of \cite{TP08}.  For other conical symplectic resolutions, our (conjectural) ring structure on
the intersection cohomology of $X$ appears to be new.  In particular, when $X$ is a slice in the affine Grassmannian,
our conjecture posits the existence of a natural ring structure on a weight space of an irreducible representation of the
Langlands dual group.  This may be related to the ring structure on an entire irreducible representation constructed by
Feigin, Frenkel, and Rybnikov \cite{FFR} (Remark \ref{ring structure}).

Section \ref{sec:statement} is devoted to the statement of our conjecture, while the remainder of the paper is 
dedicated to the proof in the hypertoric case.  The proof involves two technical results about Orlik-Terao
algebras that we believe may be of independent interest, and we therefore placed them in an appendix that can
be read independently from the rest of the paper.

\vspace{\baselineskip}
\noindent
{\em Acknowledgments:}
The authors would like to thank Roman Bezrukavnikov, 
Tom Braden, Joel Kamnitzer, Davesh Maulik, Andrei Okounkov, and Daniel Shenfeld
for helpful conversations.

\section{Statement of the conjecture}\label{sec:statement}
The purpose of this section is to make the necessary definitions to state our conjecture.

\subsection{Conical symplectic resolutions}\label{PCSR}
Let $(\tX,\omega)$ be a symplectic variety equipped with an action of $\cs$, and let $X = \Spec\C[\tX]$.
We say that $\tX$ is a {\bf conical symplectic resolution} if $\cs$ acts on $\omega$ with positive weight,
$\C[\tX]$ is non-negatively graded with only the constants in degree zero,
and the natural map from $\tX$ to $X$ is a projective resolution of singularities.
Examples of conical symplectic resolutions include the following:
\begin{itemize}
\item $\tX$ is a crepant resolution of $X = \C^2/\Gamma$, where $\Gamma$
is a finite subgroup of $\operatorname{SL}(2;\C)$.
The action of $\cs$ is induced by the inverse of the diagonal action on $\C^2$.
\item $\tX$ is the Hilbert scheme of a fixed number of points on the crepant resolution of $\C^2/\Gamma$,
and $X$ is the symmetric variety of unordered collections of points on the singular space.
\item $\tX$ and $X$ are hypertoric varieties (Section \ref{sec:hypertoric}).
\item $\tX = T^*(G/P)$ for a reductive algebraic group $G$ and a parabolic subgroup $P$,
and $X$ is the affinization of this variety.  (If $G$ is of type A, then $X$
is isomorphic to the closure of a nilpotent orbit in the Lie algebra of $G$.)
The action of $\cs$ is the inverse scaling action on the cotangent fibers.
\item $X$ is a transverse slice between Schubert varieties in the affine Grassmannian,
and $\tX$ is a resolution constructed from a convolution variety (Remark \ref{ring structure}).
\item $\tX$ and $X$ are Nakajima quiver varieties \cite{Nak94,Nak98}.
\end{itemize}

\begin{remark}
  The last class of examples overlaps significantly with each of the
  others.  The first two classes are special
  cases of quiver varieties, where the underlying graph of the quiver is the extended Dynkin diagram corresponding
  to $\Gamma$.  A hypertoric variety associated to a
  cographical arrangement is a quiver variety for the corresponding graph, but not all hypertoric varieties are of this form.  
If $G$ has type A, then $T^*(G/P)$ is a quiver variety of type A, as are slices in the affine Grassmannian for $G$
(but neither of these statements holds in other types).
\end{remark}

\subsection{BBD decomposition}
Let $G$ be a reductive algebraic group acting on $\tX$ via Hamiltonian symplectomorphisms that
commute with the action of $\cs$, and let $\bG = G \times \cs$.
Let $Z := \tX\otimes_X\tX$ be the {\bf Steinberg variety}, and let $Z_0,Z_1,\ldots,Z_r$ be its irreducible components,
with $Z_0$ being the diagonal copy of $\tX$.  
Let $$H := H^{2\dim X}_{BM}(Z; \C) = \C\{[Z_0], [Z_1], \ldots, [Z_r]\}$$
be the top degree Borel-Moore homology group of $Z$.  Then $H$ is an algebra under convolution 
with unit $[Z_0]$ \cite[2.7.41]{CG97}, 
and it acts on $H^*_\bG(\tX; \C)$.  Explicitly, the action of $[Z_i]$ is the graded $H^*_\bG(*; \C)$-linear endomorphism $L_i$
of $H^*_\bG(\tX; \C)$ given by pulling and pushing along the two projections from $Z_i$ to $\tX$.
The following results follow from \cite[\S 8.9]{CG97}; the main tool in the proof is the Beilinson-Bernstein-Deligne
decomposition theorem, applied to the map $\tX\to X$.

\begin{theorem}\label{BBD-CG}
For each pair $(S,\chi)$ consisting of a symplectic leaf of $X$ and a local system $\chi$ on $S$,
there is a vector space $V_{(S,\chi)}$ such that the following statements hold.
\begin{enumerate}
\item The convolution algebra $H$ is semisimple with $$H\;\cong\; \bigoplus_{(S,\chi)}\End\!\left(V_{(S,\chi)}\right).$$
\item Let $\becircled X$ be the dense symplectic leaf and $\operatorname{triv}$ the trivial local system
on $\becircled X$.  Then $V_{(\becircled X, \operatorname{triv})} \cong\C$.
\item There is a canonical isomorphism $$\IH \cong \Hom_H\!\left(V_{(\becircled X, \operatorname{triv})},
H^*_\bG(\tX; \C)\right).$$
\item  The kernel of the map from $H$
to $\End\!\left(V_{(\becircled X, \operatorname{triv})}\right)$ is equal to $\C\big\{[Z_1],\ldots,[Z_r]\big\}$.
\end{enumerate}
\end{theorem}

From this we may deduce the following corollary.

\begin{corollary}\label{BBD-cor}
There is a canonical decomposition of graded $H^*_\bG(\*; \C)$-modules 
$$H^*_\bG(\tX; \C) \;\cong\; \bigcap_{i=1}^r \Ker(L_i) \;\oplus\; \sum_{i=1}^r \Im(L_i),$$
and the first summand is canonically isomorphic to $\IH$.
\end{corollary}

\begin{proof}
From part (1) of Theorem \ref{BBD-CG}, we have 
$$H^*_\bG(\tX; \C) \;\cong\;\bigoplus_{(S,\chi)}\Hom_H\!\left(V_{(S,\chi)}, H^*_\bG(\tX; \C)\right)\otimes V_{(S,\chi)}.$$
From parts (2) and (3), the summand corresponding to the pair $(\becircled X, \operatorname{triv})$ 
is canonically isomorphic to $\IH\otimes\C \cong \IH$.  From part (4), the complementary summand
is equal to $\C\big\{[Z_1],\ldots,[Z_r]\big\}\cdot H^*_\bG(\tX; \C) = \sum\Im(L_i)$.
\end{proof}

\subsection{Quantum cohomology}
Let $C\subset H_2(\tX; \Z)/H_2(\tX; \Z)_{\operatorname{torsion}}$ be the semigroup of effective curve classes.
Let $$\L := \C[C] = \C\{q^\b\mid \b\in C\}$$ be the semigroup ring of $C$, 
and let $\hL$ be the completion of $\L$ at the augmentation ideal.

Assume that we are given a class $\kappa\in H^2(\tX; \Z/2\Z)$ with the property that the restriction of $\kappa$
to any smooth Lagrangian subvariety of $\tX$ is equal to its second Steifel-Whitney class.    
If $\tX$ is a cotangent bundle, this condition uniquely determines $\kappa$.  If $\tX$ is a Hamiltonian
reduction of a symplectic vector space by the linear action of a reductive group, then there is a natural choice for $\kappa$
\cite[\S 2.4]{BLPWgco}.  These two cases cover all but the fifth class of examples in Section \ref{PCSR}.\footnote{It is not known
whether such a class $\kappa$ exists in general, or whether there is always a canonical choice.}

Let $\QH$ denote the $\bG$-equivariant quantum cohomology ring of $\tX$, 
modified in the sense of \cite[\S 1.2.5]{QGQC}.  More precisely, the element $q^\beta$ in our ring corresponds
to the element $(-1)^{(\b,\kappa)}q^\beta$ in the usual quantum cohomology ring.
As a graded vector space, we have
$$\QH := H^*_\bG(\tX; \C)\otimes_\C \hL,$$
where $\hL$ lies in degree zero.
Let $\QHpol \subset \QH$ be the $\L$-subalgebra generated by the subspace $H^*_\bG(\tX; \C)\otimes_\C \L$.

Consider the maximal ideal $$\mm := \big\langle 1 - q^{\b}\mid \b\in C\big\rangle \subset \L.$$

\begin{remark}
Philosophically, $\mm$ should be the ``worst possible" maximal ideal, in that the formula for modified quantum multiplication 
for the Springer resolution \cite[1.1]{BMO}, hypertoric varieties \cite[4.2]{McBS}, and quiver varieties
\cite[1.3.2]{QGQC} all involve rational functions with denominators of the form
$1  - q^{\b}$ for some effective curve class $\b$.
\end{remark}

Consider the ring $$R_\bG(\tX) := \QHpol \otimes_{\L} \L/\mm,$$
which is a graded algebra over $\C[\hbar] = H^*_\cs\!(*; \C)$.  Let $$R'_\bG(\tX) := R_\bG(\tX)/\Ann(\hbar).$$
We are now prepared to state our main conjecture.

\begin{conjecture}\label{main}
Consider the natural map $\psi_\bG:H^*_\bG(\tX; \C)\to R'_\bG(\tX)$ of graded $\C[\hbar]$-modules given by the composition
$$H^*_\bG(\tX; \C)\hookrightarrow \QHpol \twoheadrightarrow R'_\bG(\tX).$$
\begin{enumerate}
\item The map $\psi_\bG$ is surjective.
\item The kernel of $\psi_\bG$ is equal to $\displaystyle\sum_{i=1}^r\Im(L_i)$.
\end{enumerate}
\end{conjecture}

\begin{proposition}\label{ih}
If Conjecture \ref{main} holds, then $\psi_\bG$ descends to isomorphism of graded $H^*_\bG(*; \C)$-modules
from $\IH$ to $R'_\bG(\tX)$.
\end{proposition}

\begin{proof}
By Corollary \ref{BBD-cor}, we have a canonical isomorphism $\IH \cong H^*_\bG(\tX; \C)/\sum\Im(L_i)$.
\end{proof}

\begin{remark}\label{getting rid of G}
Since $H^*(\tX; \C)$ and $H^*_\bG(*; \C)$ both vanish in odd degree \cite[2.5]{BLPWquant}, the Leray-Serre
spectral sequence for the fibration $X_\cs\hookrightarrow X_\bG\to BG$ 
tells us that $H^*_\bG(\tX; \C)$ is a free module over $H^*_G(*; \C)$ and 
$H^*_\cs\!(\tX; \C) \cong H^*_\bG(\tX; \C)\otimes_{H^*_\bG(*; \C)}\C[\hbar]$.  Similar statements hold for quantum cohomology
of $\tX$ and intersection cohomology of $X$.  For this reason, if Conjecture \ref{main} holds for $\psi_\bG$,
then it also holds for $\psi_\cs$.
\end{remark}

\begin{example}\label{basic example}
Consider $T^*\mathbb{P}^1$, equipped with the inverse scaling action of $\cs$ on the fibers and the natural action
of a maximal torus $T\subset \operatorname{PGL}(2)$.  
We have $H^*_\bT(T^*\mathbb{P}^1; \C) = \C[x,y,\hbar]/\langle xy\rangle$, where $x = [T^*_0\mathbb{P}^1]$
and $y = [T^*_\infty\mathbb{P}^1]$.  

In quantum cohomology, we have
$$x * y = \frac{q}{1-q} \hbar L_1(y).$$
Here $L_1(y) = [\mathbb{P}^1] = \hbar - x - y$, but it will not be necessary to know this for the discussion that follows.
We have presentations 
$$QH^*_\bT(T^*\mathbb{P}^1; \C) = \C[x,y,\hbar][[q]]\Big{/}\left\langle xy - \frac{q}{1-q} \hbar L_1(y)\right\rangle$$
and
$$QH^*_\bT(T^*\mathbb{P}^1; \C)_{\operatorname{pol}} = \C[x,y,\hbar,q]\Big{/}\Big\langle (1-q)xy - q\hbar L_1(y)\Big\rangle.$$
Setting $q=1$ gives us $$R_\bT(T^*\mathbb{P}^1)\cong\C[x,y,\hbar,q]\big{/}\big\langle \hbar L_1(y)\big\rangle,$$ 
and killing the annihilator of $\hbar$
gives us $$R'_\bT(T^*\mathbb{P}^1)\cong\C[x,y,\hbar,q]\big{/}\big\langle L_1(y)\big\rangle.$$
Of course, it is not the case that $\psi_\bT$ takes the class in $H^*_\bT(T^*\mathbb{P}^1; \C)$
represented by an arbitrary polynomial $f(x,y,\hbar)$ to the class in $R'_\bT(T^*\mathbb{P}^1)$ represented
by the same polynomial; this would not be well-defined.  However, $\psi_\bT$ does behave this way on linear polynomials,
and this (along with $H^*_\bT(*; \C)$-linearity) is enough to conclude both that $\psi_\bT$ is surjective and that the image
of $L_1$ is contained in the kernel.  The fact that the image of $L_1$ is equal to the kernel can be concluded by counting dimensions.
\end{example}

\begin{remark}
The quantum correction 
to multiplication by a divisor is conjecturally given by linear combinations of operators of the form
$$\frac{q^\beta}{1-q^\beta} \hbar L_\beta,$$
where $\beta\in H_2(\tX; \C)$ is one of finitely many ``roots" of $\tX$, and
$L_\beta$ is a linear combination of $L_1,\ldots,L_r$.
This conjecture is proved for the Springer resolution \cite{BMO}, as well as for quiver varieties
(modulo Kirwan surjectivity in degree 2) \cite{QGQC}.
In these situations, quantum multiplication is naturally defined on
the vector space $\QHloc := H^*_\bG(\tX; \C)\otimes_\C \Lambda_{\operatorname{loc}}$,
where $\Lambda_{\operatorname{loc}}$ is given by adjoining inverses to the classes $1-q^\beta$.
We may regard $\Spec\QHloc$ as an affine subvariety of 
$$\Spec\;\Sym_{H^*_\bG(*; \C)}\!\big(H^*_\bG(\tX; \C)\big)\times\Spec\Lambda_{\operatorname{loc}}
= H_*^\bG(\tX; \C)\times\Spec\Lambda_{\operatorname{loc}},$$
and $\Spec\QHpol$ is then the closure of $\Spec\QHloc$ inside of $H_*^\bG(\tX; \C)\times\Spec\Lambda.$
If $\Spec\QHpol$ is flat over $\Spec\Lambda$ and all of the roots are primitive, then it follows that
the image of every $L_\beta$ is contained in the kernel of $\psi_\bG$.
This gives us a strategy for proving one of the two inclusions needed for part (2) of Conjecture \ref{main}
in a number of other cases.
\end{remark}

\begin{remark}\label{ring structure}
The intersection cohomology $\IH$ is a priori only a graded $H^*_\bG(*; \C)$-module, while $R'_\bG(\tX)$ is an algebra.
One of the interesting consequences of Conjecture \ref{main} and Proposition \ref{ih} is that it would endow $\IH$
with an algebra structure.  In the case of hypertoric varieties, the module $I\! H_T^*(\tX; \C)$ was given an algebra
structure by Braden and the second author \cite{TP08} via completely different means, 
and this coincides with the algebra structure that we obtain
in this paper after setting $\hbar$ to zero (Proposition \ref{same iso}).

Another intriguing class of examples is the fifth one mentioned in Section \ref{PCSR}.
Fix a simple, simply laced algebraic group $G$ with maximal torus $T\subset G$.
Let $\operatorname{Gr}$ be the affine Grassmannian for $G$, and for
any dominant coweight $\la\in\Hom(\cs,T)$, consider the Schubert variety $\operatorname{Gr}^\la\subset\operatorname{Gr}$.
Fix dominant coweights $\la\geq\mu$, and let $X$ be a normal slice to $\operatorname{Gr}^\mu$ inside of $\operatorname{Gr}^\la$.
Using the geometric Satake correspondence \cite{G-GS,MV-GS}, 
Ginzburg produces an isomorphism between a quotient of $I\! H_T^*(X; \C)$ (obtained by choosing
generic values for the equivariant parameters)
and the $\mu$ weight space of the irreducible representation $V(\la)$ of $G^L$ \cite[3.11 \& 5.2]{G-GS}.
If $\la$ is a sum of minuscule coweights (for example, if $G$ is of type A), 
then $X$ admits a conical symplectic resolution \cite[2.9]{KWWY}.
Thus our Conjecture \ref{main} and Proposition \ref{ih} would endow the weight space $V(\la)_\mu$ with a ring structure.

In \cite{FFR}, Feigin, Frenkel, and Rybnikov define a ring structure on the whole representation $V(\la)$ using a \emph{quantum shift of argument subalgebra}. We believe that the above considerations may lead to a geometric explanation of their results.
\end{remark}

\excise{
\begin{remark}
When $X$ is a hypertoric variety, $\IH$ admits a canonical ring structure \cite{TP08}.\footnote{Strictly speaking,
we constructed the ring structure with the Hamiltonian torus action but without the $\bS$-action.  However,
it would not be difficult to extend this construction $\bS$-equivariantly.}
The construction was somewhat opaque, in that we gave a combinatorial proof that there exists a unique ring structure with certain
functorial properties, but we gave no geometric explanation for the existence of such a ring structure.

It is easy to check that the presentation of $\IH$ in \cite{TP08} matches the presentation
of $R'(\tX)$ obtained from \cite{McBS}, thus Proposition \ref{ih} ``explains" where the ring structure in \cite{TP08}
came from.  More generally, Proposition \ref{ih} imposes a ring structure on the equivariant intersection cohomology
of any affine cone that admits a conical symplectic resolution.
\end{remark}
}

\section{Hypertoric varieties}\label{sec:hypertoric}
In this section we prove Conjecture \ref{main} for hypertoric varieties.

\subsection{Definitions}\label{sec:defs}
We begin by reviewing the constructive definition of a projective hypertoric variety, which was first introduced in \cite{BD}.
An intrinsic approach to hypertoric varieties can be found in \cite{AP}.

Fix a finite-rank lattice $N$, along with a list of (not necessarily distinct) nonzero primitive vectors $a_1,\ldots,a_n\in N$ and
integers $\th_1,\ldots,\th_n$.  Consider the hyperplanes
$$H_i := \{x\in N_\R^\vee\mid \langle a_i, x\rangle + \th_i = 0\}$$
along with the associated half-spaces
$$H_i^+ := \{x\in N_\R^\vee\mid \langle a_i, x\rangle + \th_i \geq 0\}
\and H_i^- := \{x\in N_\R^\vee\mid \langle a_i, x\rangle + \th_i \leq 0\}.$$
We make the following assumptions on our data:
\begin{itemize}
\item {\bf Full rank:} The lattice $N$ is spanned by $\{a_1,\ldots,a_n\}$.
\item {\bf No co-loops:} For all $i$, the lattice $N$ is spanned by $\{a_1,\ldots,a_n\}\smallsetminus\{a_i\}$.
\item {\bf Unimodular:}  For any $S\subset[n]$, if $\{a_i\mid i\in S\}$ spans $N_\Q$ over $\Q$, then it spans $N$ over $\Z$.
\item {\bf Simple:}  For any $S\subset[n]$, $\displaystyle\codim\bigcap_{i\in S}H_i = |S|$ (note that the empty set has every codimension).
\end{itemize}
Consider the short exact sequence
$$0\longrightarrow P\overset{\iota}\longrightarrow\Z^n\overset{\pi}{\longrightarrow} N\longrightarrow 0,$$
where $\pi$ takes the $\ith$ coordinate vector to $a_i$ and $P := \ker(\pi)$.
Dualizing and then taking homomorphisms into $\cs$, we obtain an exact sequence of tori
$$1\to K\to T^n\to T\to 1.$$  The torus $T^n$ acts symplectically on $T^*\C^n$ with moment map
$$\mu_n:T^*\C^n\to \Lie(T^n)^\vee\cong\C^n$$ given by the formula
$$\mu_n(z_1,w_1,\ldots,z_n,w_n) = (z_1w_1,\ldots,z_nw_n).$$
Composing with $\iota^\vee:\C^n\to P^\vee_\C$, we obtain a moment map $$\mu_K:T^*\C^n\to P^\vee_\C$$
for the action of $K$.  The element $\theta = (\th_1,\ldots,\th_n)\in\Z^n\cong\Hom(T^n,\cs)$ is a character of $T^n$,
which we also regard as a character of $K$ by restriction.
Consider the symplectic quotients $$X:= \mu_K^{-1}(0)\mmod_{\!0} \,K = \Spec\C[\mu_K^{-1}(0)]^K,$$
and $$\tX := \mu_K^{-1}(0)\mmod_{\!\th} \,K = \Proj\left(\C[\mu_K^{-1}(0)]\otimes\C[t]\right)^K,$$
where $K$ acts on $t$ via the character $\th$.  The assumptions of simplicity and unimodularity
imply that the natural map from $\tX$ to $X$ is a projective symplectic resolution \cite[3.2 \& 3.3]{BD}.

The action of $\cs$ on $T^*\C^n$ via inverse scaling of the cotangent fibers descends to an action on $\tX$,
and the symplectic form has weight 1 with respect to this action.
The assumption of no co-loops implies that $\C[X]^\cs = \C$, and therefore that $\tX$ is a
conical symplectic resolution of $X$.
The Hamiltonian action of $T^n$ on $T^*\C^n$ induces an action on $\tX$, and this descends
to an effective Hamiltonian action of $T$ that commutes with the action of $\cs$.  Let $\bT = T \times \cs$.

\subsection{Cohomology}
We next review some basic facts about the cohomology of hypertoric varieties.
A minimal set $C\subset[n]$ such that $\bigcap_{i\in C}H_i = \emptyset$ is called a {\bf circuit}.
If $C$ is a circuit, then there exists a unique decomposition $$C = C^+\sqcup C^-\qquad\text{such that}\qquad
\bigcap_{i\in C^+}H_i^+ \cap \bigcap_{i\in C^-}H_i^- = \emptyset.$$

Let $A := \Sym N^\vee \cong H^*_{T}(*; \Z)$ be the $T$-equivariant cohomology ring of a point.
The $\bT$-equivariant cohomology ring of $\tX$ was computed by Harada and the second author \cite[4.4]{HP04}, building
on the $T$-equivariant computation in \cite{Ko99}.

\begin{theorem}\label{cohom}
The ring $H^*_{\bT}(\tX; \Z)$ is isomorphic to $\Z[u_1,\ldots,u_n,\hbar]/J_0$,
where $J_0$ is the ideal generated by
$$\prod_{i\in C^+}u_i\cdot\prod_{j\in C^-}(\hbar-u_j)$$
for each circuit $C\subset[n]$.  We have $\deg(u_i) = \deg(\hbar) = 2$ for all $i$, and
the $A$-algebra structure is given by the natural inclusion 
$$\pi^\vee:N^\vee\to\Z^n\cong\Z\{u_1,\ldots,u_n\}.$$
\end{theorem}

\begin{corollary}
We have a canonical isomorphism $H_2(\tX; \Z)\cong P$.
\end{corollary}

\begin{proof}
Since the generators of $J$
all have degree at least 4, we have 
$$H^2_{\bT}(\tX; \Z) \cong \Z\{u_1,\ldots,u_n,\hbar\} = \Z^n\oplus\Z,$$
and therefore $H^2(\tX; \Z)\cong \Z^n/N^\vee \cong P^\vee$.
Dualizing, we obtain our result.
\end{proof}

For any set $S\subset [n]$, let $u_S := \prod_{i\in S}u_i$.  The set $S$ is called {\bf independent} if it contains no circuits.

\begin{corollary}\label{small monomials}
The ring $H^*_{\bT}(\tX; \Z)$ is spanned over $A[\hbar]$ by 
monomials of the form $u_S$, where $S\subset[n]$ is independent.
\end{corollary}

\begin{proof}
It is sufficient to prove that $H^*_{T}(\tX; \Z)$ is spanned over $A$ by 
monomials of the form $u_S$, where $S\subset[n]$ is independent.
This is shown in the appendix (Lemma \ref{independence}).
\end{proof}

\subsection{Quantum cohomology}
We continue by describing the various versions of the quantum cohomology ring of $\tX$.
For any circuit $C\subset[n]$, unimodularity implies that $$\sum_{i\in C^+}a_i - \sum_{j\in C^-}a_j = 0.$$
Let $C_i = 1$ if $i\in C^+$, $-1$ if $i\in C^-$, and 0 otherwise, and
consider the element $$\b_C := \sum_{i=1}^n C_i e_i \in \ker(\Z^n\to N) = P \cong H_2(\tX; \Z).$$
To compute $\QHTpol$, we need the following lemma, which is implicit in \cite{McBS}. 

\begin{lemma}\label{small}
For any independent subset $S\subset[n]$, the quantum product of $\{u_i\mid i\in S\}$ is equal to 
the ordinary product $u_S$.
\end{lemma}

\begin{proof}
We proceed by induction on the size of $S$.  Let $j$ be the maximal element of $S$, and let $\bar{S} = S\smallsetminus\{j\}$.
By our inductive hypothesis, the quantum product of the elements $\{u_i\mid i\in S\}$ is equal to the quantum
product of $u_{j}$ with $u_{\bar{S}}$.  By \cite[4.2]{McBS}, we have\footnote{The formula in \cite{McBS} has $q^{\b_C}$
replaced with $(-1)^{|C|}q^{\b_C}$ because that paper uses the unmodified quantum product.}
\begin{equation*} \label{divisorformula} u_{j} \cdot u_{\bar{S}} = u_S + \hbar \sum_C C_{j} \frac{q^{\b_C}}{1-q^{\b_C}} L_C(u_{\bar{S}}), \end{equation*}
where $L_C$ is a certain linear combination of $L_1,\ldots,L_r$.
Thus it is sufficient to show that $L_C(u_{\bar{S}}) = 0$ for all circuits $C$ containing $j$.

Let $\mu : \tX \to N^\vee_\C$ be the moment map induced by $\mu_n$ for the action of $T$ on $\tX$.  
The operator $L_C$ is given by a correspondence $Z_C\subset Z = \tX\times_X\tX$
that lies over the locus $$H_C := \bigcap_{i\in C} H^\C_i \subset N^\vee_\C.$$
On the other hand, the element $u_{\bar{S}}$ may be represented by a cycle that lies over $H_{\bar{S}}$,
thus $L_C(u_{\bar{S}})$ may be represented by a cycle that lies over $H_{C\cup \bar{S}} \subset H_S$.
Since $S$ is independent, we have $$\codim H_{S} = |S| > |\bar{S}| = \frac 1 2 \deg u_{\bar{S}} = \frac 1 2 \deg L_C(u_{\bar{S}}),$$
which implies that $L_C(u_{\bar{S}}) = 0$.
\end{proof}

\begin{theorem}\label{QHTpol}
The ring $\QHTpol$ is isomorphic to $\Lambda[u_1,\ldots,u_n,\hbar]/J$,
where $J$ is the ideal generated by
$$\prod_{i\in C^+} u_i\cdot \prod_{j\in C^-} (u_j-\hbar) \;\;-\;\;
q^{\b_C}\prod_{i\in C^+} (u_i - \hbar)\cdot \prod_{j\in C^-} u_j$$
for each circuit $C\subset[n]$.  The $A_\C$-algebra structure is as in Theorem \ref{cohom}. 
\end{theorem}

\begin{proof}
The fact that $\QHTpol$ is generated over $\Lambda$ by $H^2_{\bT}(\tX; \C)$
follows from Corollary \ref{small monomials} and Lemma \ref{small}.
The rest of the theorem appears in \cite[1.1]{McBS}.
\end{proof}

\begin{corollary}\label{RT}
The ring $R_\bT(\tX)$ is isomorphic to $\C[u_1,\ldots,u_n,\hbar]/J_1$, where $J_1$ is the ideal generated by
$$\prod_{i\in C^+} u_i\cdot \prod_{j\in C^-} (u_j-\hbar) \;\;-\;\;
\prod_{i\in C^+} (u_i - \hbar)\cdot \prod_{j\in C^-} u_j$$
for each circuit $C\subset[n]$.
%
The ring $R'_{\bT}(\tX):=R_\bT(\tX)/\Ann(\hbar)$ is isomorphic to $\C[u_1,\ldots,u_n,\hbar]/J_1'$, where $J_1'$ is the ideal generated by
$$\fch:=\hbar^{-1}\!\left(\prod_{i\in C^+} u_i\cdot \prod_{j\in C^-} (u_j-\hbar) \;\;-\;\;
\prod_{i\in C^+} (u_i - \hbar)\cdot \prod_{j\in C^-} u_j\right)$$
for each circuit $C\subset[n]$.
\end{corollary}

\excise{
\begin{remark}
It is a nontrivial fact, which we will need later, that $R'_{\bT}(\tX)$ is free over $\C[\hbar]$; see
Theorem \ref{thm-apx}.
\end{remark}
}

\begin{proposition}\label{small quantum monomials}
The ring $R'_{\bT}(\tX)$ is spanned over $A_\C[\hbar]$ by monomials of the form $u_S$,
where $S\subset[n]$ is independent.
\end{proposition}

\begin{proof}
It is sufficient to prove that $R'_{\bT}(\tX)/\langle\hbar\rangle$ is spanned over $A_\C$ by 
monomials of the form $u_S$, where $S\subset[n]$ is independent.
This is shown in the appendix (Theorem \ref{monomial span} and Remark \ref{same OT}).
\end{proof}


\begin{corollary}\label{surj}
The map $\psi_\bT$ from Conjecture \ref{main} is surjective.
\end{corollary}

\begin{proof}
Lemma \ref{small} tells us that $\psi_\bT$ takes the image of $u_S$ in $\QHTpol$
to the image of $u_S$ in $R'_{\bT}(\tX)$ for all independent $S\subset [n]$.  
By Proposition \ref{small quantum monomials}, this implies that $\psi_\bT$
is surjective.
\end{proof}

\subsection{The kernel of \boldmath{$\psi_\bT$}}
In this section we prove the second half of Conjecture \ref{main} for hypertoric varieties.
Let $$U\;\; := \;\;\Ker(\psi_\bT) \;\;\subset\;\; H^*_\bT(\tX; \C) \;\;\supset\;\; \Im(L_1)+\ldots+\Im(L_r)\;\; =:\;\; V;$$
the conjecture says that $U = V$.  

For any circuit $C$, let $\Cbar$ be the set obtained from $C$ by deleting
the maximal element $\imin\in C$, and consider the graded vector subspace
$$W := A_\C[\hbar]\cdot \big\{u_S \fch \;\;\big{|}\;\; \text{$C$ is a circuit, $S\cap C = \emptyset$, and $S\cup\Cbar$ is independent}\big\}.$$

\begin{lemma}\label{WinU}
$W\subset U$.
\end{lemma}

\begin{proof}
Each term of $u_S \fch$ is equal to plus or minus a power of $\hbar$ times a square-free monomial of independent
support.  By Lemma \ref{small}, such a monomial is taken to itself by $\psi_\bT$.  This means that
$\psi_\bT$ takes $u_S \fch$ to itself, and $u_S \fch$ represents the zero element of $R'_\bT(\tX)$ by Corollary \ref{RT}.
\end{proof}

\begin{lemma}\label{WinV}
$W\subset V$.
\end{lemma}

\begin{proof}
Fix a circuit $C$, and assume that $\imin\in C^+$.  Let
$$g_C \;\;:=  \prod_{i \in C^+ \setminus \{\imin\}} \!\!u_i \,\,\cdot\prod_{j \in C^-} (u_j - \hbar)\;\;\in\;\; 
H^*_\bT(\tX; \C)\;\;\subset\;\; \QHTpol.$$
(Note that, by Lemma \ref{small}, this product has no quantum correction.)
By Theorem \ref{QHTpol}, we have 
\[ u_{\imin}\; g_C \;\;=\;\; \prod_{i \in C^+} u_i \,\,\cdot\prod_{j \in C^-} (u_j - \hbar) \;\;=\;\;q^{\beta_C} \prod_{i \in C^+} (u_i - \hbar) \,\,\cdot \prod_{j \in C^-} u_j \;\;\in\;\; \QHTpol. \]
By definition of $\fch$, we have
\begin{eqnarray*}q^{\b_C}\hbar \fch &=& q^{\b_C}\prod_{i\in C^+} u_i\cdot \prod_{j\in C^-} (u_j-\hbar) \;\;-\;\;
q^{\b_C}\prod_{i\in C^+} (u_i - \hbar)\cdot \prod_{j\in C^-} u_j\\
&=& q^{\b_C}u_{\imin}\; g_C - u_{\imin}\, g_C\\
&=& \big(q^{\b_C}-1\big)u_{\imin}\, g_C \;\;\in\;\; \QHTpol.
\end{eqnarray*}
Fix a set $S$ such that $S\cap C=\emptyset$ and $S\cup\Cbar$ is independent.  
Multiplying both sides of the above equation by $u_S$, we obtain
\begin{equation*}\label{middle}
q^{\beta_C} \hbar\, u_S \fch \;\; =\;\; \big(q^{\beta_C}-1\big) u_{\imin} u_S g_C.
\end{equation*}
Since the classical product of $u_{\imin}$ with $g_C$ vanishes and $u_S g_C\in H^*_\bT(\tX; \C)$ by Lemma \ref{small}, 
we have \cite[4.2]{McBS}
\begin{equation*} \label{circuitsum} u_{\imin} \cdot u_S g_C\;\; =\;\;  \hbar \sum_{D} D_{\imin} \frac{q^{\beta_D}}{1-q^{\beta_D}} L_D(u_S g_C), \end{equation*}
where $D$ ranges over all circuits. 

Let $\C(\L)$ be the field of fractions of $\L$, and let $\QHrat$ be the ring generated by $H^*_\bG(\tX; \C)\otimes_\C \C(\L)$ under the quantum product. It follows easily from \cite[4.2]{McBS} that in fact $\QHrat = H^*_\bG(\tX; \C)\otimes_\C \C(\L)$ as a vector space,
and that $\QHpol \subset \QHrat$. We may think of elements of $\QHrat$ as meromorphic sections of the vector bundle with fiber $H^*_\bG(\tX; \C)$ over $\operatorname{Spec} \L$. In particular they have well-defined loci of poles.

We can now combine the two above equations 
to obtain \[ \frac{q^{\beta_C}}{q^{\beta_C}-1} \hbar\, u_S \fch \;\;=\;\; 
\hbar \sum_{D} D_{\imin} \frac{q^{\beta_D}}{1-q^{\beta_D}} L_D(u_S g_C) \;\;\in\;\;\QHrat.\]
Since the left-hand side has poles only at $q^{\beta_C}=1$, so does the right-hand side. 
We conclude that all summands such that $D\neq C$ vanish, and we are left with
\begin{equation*}\label{the third one}\frac{q^{\beta_C}}{q^{\beta_C}-1} \hbar\, u_S \fch \;\;=\;\; \hbar\frac{q^{\beta_C}}{1-q^{\beta_C}} L_C(u_S g_C).\end{equation*}
Dividing by $\frac{q^{\beta_C}}{1-q^{\beta_C}}$, we have $$\hbar\, u_S \fch \;\;=\;\; -\hbar L_C(u_S g_C).$$
{\em A priori}, this equation lives in $\QHrat$.  However, it is clear that both sides live
in the subspace $\QHpol\subset\QHrat$.  Furthermore, since $u_S f_{C}$ is a sum of powers of $\hbar$
times independent square-free monomials, Lemma \ref{small} tells us that $u_S \fch$ lies in $H^*_\bT(\tX; \C)
\subset\QHpol$.  
Since $H^*_\bT(\tX; \C)$ is a free module over $\C[\hbar]$,
we may divide by $\hbar$ to obtain $$u_S \fch \;\;=\;\; -L_C(u_S g_C)\;\;\in\;\; H^*_\bT(\tX; \C).$$
Thus we see that $u_S \fch$ is in the image of $L_C$, and is therefore in the span of the images of $L_1,\ldots,L_r$.
A similar argument can be applied if $\imin\in C^-$.
\end{proof}

For any $\mathbb{N}$-graded vector space $Y = \bigoplus Y^k$ with finite-dimensional graded pieces,
let $$\Hilb(Y; t) := \sum_{k=0}^\infty \dim Y^k t^k\in\mathbb{N}[[t]].$$

\begin{lemma}\label{HilbUV}
$\Hilb(U; t) = \Hilb(V; t)$.
\end{lemma}

\begin{proof}
Since $\psi_\bT$ is surjective (Corollary \ref{surj}), we have $R'_\bT(\tX) \cong H^*_\bT(\tX; \C)/U$.
By Corollary \ref{BBD-cor}, we also have $\IHT \cong H^*_\bT(\tX; \C)/V$.  Thus the
statement that $\Hilb(U; t) = \Hilb(V; t)$ is equivalent to the statement that $\Hilb(R'_\bT(\tX); t) = \Hilb(\IHT; t)$.

By Remark \ref{same OT} and Theorem \ref{thm-apx}, we have $$\Hilb(R'_\bT(\tX); t) = \Hilb(\OTh; t)
= (1-t)^{-1}\Hilb(\OT; t).$$
By Proposition \ref{ranks}, $\Hilb(\OT; t)$ is equal to $(1-t)^{-\rk N}$ times the $h$-polynomial of the broken circuit complex
of the matroid represented by the vectors $a_1,\ldots,a_n$.
On the other hand, we have $$\Hilb(\IHT; t) = (1-t)^{-\rk N - 1}\Hilb(\IHz; t),$$
and $\Hilb(\IHz; t)$ is itself equal to the $h$-polynomial of the broken circuit complex \cite[4.3]{PW07}.
Thus $\Hilb(U; t) = \Hilb(V; t)$.
\end{proof}

Let $V_0 = V \otimes_{\C[\hbar]}\C$, and let $W_0$ be the image of $W\subset V$ in $V_0$.
More concretely, $V_0$ is the complement of $I\! H^*_T(X; \C)$ in 
$$H^*_T(\tX; \C) \cong \C[u_1,\ldots,u_n]/\langle u_C\mid \text{$C$ a circuit}\rangle,$$ and $W_0$
is the $A_\C$-submodule of $H^*_T(\tX; \C)$ spanned by $\{u_Sf_{C,0}\}$, where $f_{C,0}$ is obtained
from $\fch$ by setting $\hbar$ equal to zero.


\begin{lemma}\label{what's in W0}
Let $C$ be a circuit and let $S$ be a set disjoint from $C$.
For any collections of non-negative integers $\underline{d} = (d_i\mid i\in S)$ and $\underline{e} = (e_j\mid j\in\Cbar)$,
we have
$$f(S,C,\underline{d},\underline{e}) \;\; :=\;\; 
u_S f_C\cdot\prod_{i\in S}u_i^{d_i}\cdot\prod_{j\in\Cbar}(C_ju_j - C_{\imin}u_{\imin})^{e_j}\;\;\in\;\; W_0.$$
\end{lemma}

\begin{proof}
We proceed via a double induction.
First, we fix $C$ and induct downward on the size of $S$.
If $|S| > \rk N - |\Cbar|$, then every term of $u_S f_C$ contains a monomial supported on
a dependent set, so $u_S f_C = 0$.
Thus we will fix $S$ and assume that the lemma holds for all sets $S'\supsetneq S$ disjoint from $C$.
By the same reasoning, we may assume that $S\cup \Cbar$ is independent.

Second, we induct upward on the exponents.  The base case is where $d_i = 0 = e_j$ for all $i$ and $j$,
in which case $f(S,C,\underline{d},\underline{e})= u_S f_C \in W_0$ by definition of $W_0$.  Thus we may
fix $\underline{d}$ and $\underline{e}$ such that $f(S,C,\underline{d},\underline{e})\in W_0$ and prove
that for all $i\in S$ and $j\in \Cbar$, we have
$u_i f(S,C,\underline{d},\underline{e})\in W_0$ and $(C_ju_j - C_{\imin}u_{\imin}) f(S,C,\underline{d},\underline{e})\in W_0$.

Let $i\in S$ be given.  Since $S\cup\Cbar$ is independent, there exists $x\in N^\vee$ such that $\pi^\vee(x) = \sum \gamma_k u_k$
with $\gamma_i = 1$ and $\gamma_k = 0$ for all $k\in S\cup C\smallsetminus\{i\}$.
Then 
\begin{eqnarray*}\pi^\vee(x)\cdot f(S,C,\underline{d},\underline{e}) &=& \sum_{k=1}^n \gamma_k u_k\cdot f(S,C,\underline{d},\underline{e})\\
&=& u_i \cdot f(S,C,\underline{d},\underline{e}) + \sum_{k\notin S\cup C}\gamma_k u_k\cdot f(S,C,\underline{d},\underline{e})\\
&=& u_i \cdot f(S,C,\underline{d},\underline{e}) + \sum_{k\notin S\cup C}\gamma_k \cdot f(S\cup\{k\},C,\underline{d},\underline{e}).
\end{eqnarray*}
Our first inductive hypothesis tells us that $ f(S\cup\{k\},C,\underline{d},\underline{e})\in W_0$, and $W_0$ is by definition closed under
multiplication by elements of $A$, so we also have $\pi^\vee(x)\cdot f(S,C,\underline{d},\underline{e})\in W_0$.
This implies that $u_i \cdot f(S,C,\underline{d},\underline{e})\in W_0$.

Let $j\in \Cbar$ be given.  Since $S\cup\Cbar$ is independent, there exists $x\in N^\vee$ such that $\pi^\vee(x) = \sum \gamma_k u_k$
with $\gamma_j = C_j$, $\gamma_{\imin} = -C_{\imin}$, and $\gamma_k = 0$ for all $k\in S\cup C\smallsetminus\{j,\imin\}$.
Then 
\begin{eqnarray*}\pi^\vee(x)\cdot f(S,C,\underline{d},\underline{e}) &=& \sum_{k=1}^n \gamma_k u_k\cdot f(S,C,\underline{d},\underline{e})\\
&=& (C_ju_i - C_{\imin}u_{\imin})\cdot f(S,C,\underline{d},\underline{e}) + \sum_{k\notin S\cup C}\gamma_k u_k\cdot f(S,C,\underline{d},\underline{e})\\
&=& (C_ju_i - C_{\imin}u_{\imin}) \cdot f(S,C,\underline{d},\underline{e}) + \sum_{k\notin S\cup C}\gamma_k \cdot f(S\cup\{k\},C,\underline{d},\underline{e}).
\end{eqnarray*}
By the same reasoning as above, this implies that 
$(C_ju_i - C_{\imin}u_{\imin}) \cdot f(S,C,\underline{d},\underline{e})\in W_0$.
\end{proof}

\begin{lemma}\label{HilbVW}
$V = W$.
\end{lemma}

\begin{proof}
We will start by proving that $\Hilb(W_0; t) = \Hilb(V_0; t)$.
Consider the degree-lexicographic monomial order on $H^*_T(\tX; \C)$ with $u_1 > u_2 > \ldots > u_n$.
Given $C$, $S$, $\underline{d}$, and $\underline{e}$ as in Lemma \ref{what's in W0},
The initial term of $f(S,C,\underline{d},\underline{e})$ with respect to this order is 
$\pm u_{S\cup\Cbar}\prod_S u_i^{d_i}\prod_{\Cbar} u_j^{e_j}$.
These monomials span the kernel of the projection
$$H^*_T(\tX; \C) \cong \C[u_1,\ldots,u_n]/\langle u_C\mid \text{$C$ a circuit}\rangle
\to \C[u_1,\ldots,u_n]/\langle u_{\Cbar}\mid \text{$C$ a circuit}\rangle
=: \SRbc,\footnote{This is the Stanley-Reisner ring of the broken circuit complex; see Section \ref{sec:spanning}.}$$
thus Lemma \ref{what's in W0} tells us that $\init(W_0)$ contains this kernel.
We therefore have
\begin{eqnarray*}
\Hilb(H^*_T(\tX; \C)/W_0; t) &=& \Hilb(H^*_T(\tX; \C)/\init(W_0); t)\\
&\leq& \Hilb(\SRbc; t)\\ 
&=& \Hilb(I\! H^*_T(X; \C))\;\;\;\;\text{by \cite[4.3]{PW07}}\\
&=& \Hilb(H^*_T(\tX; \C)/V_0; t).
\end{eqnarray*}
Since $W_0\subset V_0$, this implies that $W_0 = V_0$.

We would like to use this to conclude that $W = V$.  Suppose not, and let $v\in V$ be a homogeneous element of minimal degree
that is not contained in $W$.  Let $v_0$ be the image of $v$ in $V_0$.  Since $W_0 = V_0$, there exists a homogeneous
$w\in W$ such that $w_0 = v_0$.  This means that $v-w$ is in the kernel of the projection from $V$ to $V_0$, 
so there exists a homogeneous $v'\in V$
with $v-w = \hbar v'$.  By minimality of the degree of $v$, we have $v'\in W$, and therefore $v = \hbar v' + w\in W$, 
which is a contradiction.
\end{proof}

\begin{corollary}\label{all the same}
$U = V$.
\end{corollary}

\begin{proof}
Lemmas \ref{WinU} and \ref{HilbVW} imply that $V\subset U$, and they have the same Hilbert series by Lemma \ref{HilbUV},
so they must be equal.
\end{proof}

\begin{theorem}\label{hypertoric main}
Conjecture \ref{main} holds for hypertoric varieties.
\end{theorem}

\begin{proof}
The first part of the conjecture is Corollary \ref{surj}, while the second is
Corollary \ref{all the same}.
\end{proof}

\subsection{Comparison with previous work}\label{sec:BP}
Let $\OT := R'_\bT(\tX)/\langle\hbar\rangle$; this algebra is called the {\bf Orlik-Terao algebra}.
In an earlier paper, Braden and the second author showed that $I\! H_T^*(X; \C)$ is canonically
isomorphic to $\OT$ \cite[4.5]{TP08}.  The first thing we want to establish is that the isomorphism
in this paper is the same as the isomorphism in that paper.

\begin{proposition}\label{same iso}
The isomorphism from $I\! H_T^*(X; \C)$ to $\OT$ induced by $\psi_\bT$ (after setting $\hbar$ equal to zero)
coincides with the isomorphism in \cite{TP08}.
\end{proposition}

\begin{proof}
Let $F\subset [n]$ be a flat of the matroid represented by the vectors $a_1,\ldots,a_n$.
Working only with the vectors $\{a_i\mid i\in F\}$, we obtain an algebra $\OT_F$ which is isomorphic
to the quotient of $\OT$ by the ideal generated by $\{u_i\mid i\notin F\}$.
We also obtain a hypertoric variety $X_F$ which is a ``normal slice" to a stratum of $X$;
in particular, this means that an analytic neighborhood of the cone point of $X_F$ admits a normally nonsingular
inclusion into $X$ \cite[2.5]{PW07}.  Since $X_F$ is a cone, we may assume that our analytic neighborhood
is equivariantly (with respect to the maximal compact subtorus of $T_F$) homeomorphic to $X_F$ itself. 
This inclusion is easily seen to be equivariant with respect to
the maximal compact subtorus of $T_F \subset T$, and therefore induces a map
$$I\! H^*_T(X; \C)\to I\! H^*_{T_F}(X; \C) \to I\! H^*_{T_F}(X_F; \C).$$  
The isomorphisms in \cite{TP08} are the unique isomorphisms
such that the diagrams 
\[\tikz[->, thick]{
\matrix[row sep=10mm,column sep=10mm,ampersand replacement=\&]{
\node (a) {$I\! H^*_T(X; \C)$}; \& \node (c) {$\OT$};\\
\node (b) {$I\! H^*_{T_F}(X_F; \C)$}; \& \node (d) {$\OT_F$};\\
};
\draw (a) -- (c);
\draw (a) -- (b);
\draw (b) -- (d);
\draw (c) -- (d);
}\] 
commute for all $F$.  Thus it is sufficient to show that this diagram commutes using the isomorphisms
constructed in this paper for the horizontal arrows.
This follows from the fact that the inclusion of $X_F$ into $X$ lifts to an inclusion of $\tX_F$ into $\tX$ \cite[2.5]{PW07},
and the induced map from $H^*_T(\tX; \C)$ to $H^*_{T_F}(\tX_F; \C)$ is given by setting $u_i$ to zero for all $i\notin F$.
\end{proof}

We conclude by discussing some of the advantages and disadvantages of the two approaches.
The main advantage of \cite{TP08} is that the ring structure is defined at a higher categorical level:
it is shown there that the intersection cohomology sheaf $\operatorname{IC}_X$ admits
the structure of a ring object in the equivariant derived category of constructible sheaves on $X$,
and that the isomorphism from $I\! H_T^*(X; \C)$ to $\OT$ is compatible with this structure.

On the other hand, there are two advantages to the approach we take in this paper.  The first is that we work $\bT$-equivariantly
rather than $T$-equivariantly.  This may not seem like a big deal, but it is not so easy to modify the techniques
of \cite{TP08} to account for the extra $\cs$-action.  Any attempt in this direction would have to begin
with a proof of Theorem \ref{thm-apx}.  

The second, and more significant, advantage of our approach is that
the isomorphism in \cite{TP08} comes out of nowhere: one simply shows that the ring $\OT$ has the same Hilbert
series and functorial properties as $I\! H^*_T(X; \C)$, and that these functorial properties are sufficiently rigid to ensure
that the two groups are canonically isomorphic.  In contrast, the isomorphism in this paper is induced
by the natural map $\psi_\bT$, and can be (at least conjecturally)
generalized to arbitrary conical symplectic resolutions.

\appendix
\section{The Orlik-Terao algebra}
In this paper we have required two technical results 
about the Orlik-Terao algebra of a collection of vectors (Theorems \ref{monomial span} and \ref{thm-apx}).
Since we believe that these two statements may be of general interest in the theory of hyperplane arrangements, we
put them in an appendix which may be read independently from the rest of the paper.

\subsection{A spanning set}\label{sec:spanning}
Let $k$ be a field of characteristic zero.  
Let $V$ be a vector space over $k$, and let $a_1,\ldots,a_n$
be nonzero linear functions on $V$ that span $V^*$.
Let $I\subset k[u_1,\ldots,u_n]$ be the kernel of the map
taking $u_i$ to $a_i^{-1}$.  The graded $k$-algebra $\OT := k[u_1,\ldots,u_n]/I$
is called the {\bf Orlik-Terao} algebra.

For any subset $S\subset [n]$, let $u_S := \prod_{i\in S} u_i$.
A set $C\subset [n]$ is called {\bf dependent} if there exist constants $\{\eta_i\mid i\in C\}$, not all zero, such that 
$\sum \eta_ia_i = 0$.  In this case, we have a nontrivial element
$$\fcz := \sum_{i\in C}\eta_i u_{C\smallsetminus\{i\}}\in I.$$
This notation is somewhat sloppy, as $\fcz$ depends not only on $C$, but also on the constants $\eta_i$.
However, if $C$ is a {\bf circuit} (a minimal dependent set), then the constants are determined
up to a global nonzero scalar, thus the same is true for $\fcz$.

Note that if our collection of vectors is unimodular and $C$ is a circuit, then we may take $\eta_i=\pm 1$
for all $i$, and then $\fcz$ will be the polynomial obtained from the polynomial $\fch$ of Corollary \ref{RT}
(and also of Section \ref{sec:def})
by setting $\hbar$ equal to zero; this explains our funny notation.

The following result is proved in \cite[Theorem 4]{PS}.

\begin{theorem}\label{grob}
The set $\{\fcz\mid \text{$C$ a circuit}\}$ is a universal Gr\"obner basis for $I$.
\end{theorem}

For any circuit $C$, let $\Cbar$ be the set obtained from $C$ by deleting the maximal element.
Let $$\SRind := k[u_1,\ldots,u_n]/\langle u_C\mid \text{$C$ a circuit}\rangle
\and
\SRbc := k[u_1,\ldots,u_n]/\langle u_{\Cbar}\mid \text{$C$ a circuit}\rangle.$$
These algebras are called the {\bf Stanley-Reisner rings} of the independence complex
and the broken circuit complex, respectively.  Note that $u_{\Cbar}$ is (up to scale)
the initial term of $\fcz$, hence Theorem \ref{grob} says exactly that $\SRbc$ is a flat degeneration of $\OT$.

Consider the map $\Sym(V)\to k[u_1,\ldots,u_n]$ taking $v\in V$ to $\sum a_i(v)u_i$.  This makes
the algebras $\OT$, $\SRind$, and $\SRbc$ into graded $\Sym(V)$-algebras.  The following result is proved
in \cite[Propositions 1 \& 7]{PS}.

\begin{proposition}\label{ranks}
The rings $\OT$, $\SRind$, and $\SRbc$
are free as graded $\Sym(V)$-modules.
The graded rank of $\SRind$ is 
given by the $h$-numbers of the independence complex,
while the graded ranks of $\OT$ and $\SRbc$ are
given by the $h$-numbers of the broken circuit complex.
\end{proposition}

The main result of this subsection is the following.

\begin{theorem}\label{monomial span}
The ring $\OT$ is spanned over $\Sym(V)$ by elements of the form $u_S$
where $S\subset [n]$ is independent.
\end{theorem}

We begin by proving the analogous statement for $\SRind$ and $\SRbc$,
which will be used in the proof of Theorem \ref{monomial span}.

\begin{lemma}\label{independence}
The rings $\SRind$ and $\SRbc$ are spanned over $\Sym(V)$ by elements of the form $u_S$
where $S\subset [n]$ is independent.\footnote{The proof of this lemma does not require
$k$ to be a field; in particular, it holds over the integers.}
\end{lemma}

\begin{proof}
First note that $\SRbc$ is a quotient of $\SRind$, so it is sufficient to prove the lemma
only for $\SRind$.
Since $u_S$ vanishes whenever $S$ contains a circuit, it is sufficient to prove that $\SRind$
is spanned over $\Sym(V)$ by square-free monomials.  This is equivalent to showing that $\SRind\otimes_{\Sym(V)}k$
is spanned over $k$ by square-free monomials.

Consider an arbitrary monomial $u^\sigma$ for some $\sigma\in\N^n$
with independent support.
This means that there exists a set $B\subset[n]$ containing the support of $\sigma$
such that $\{a_i\mid i\in B\}$ is a basis for $V^*$.
If $\sigma_i\leq 1$ for all $i$, then we are already done, so let us suppose that there exists
an index $i\in[n]$ for which $\sigma_i>1$.  Consider the element $v\in V$ that pairs to 1
with $a_i$ and to 0 with $a_j$ for all $j\in B\smallsetminus\{i\}$, and let $u_v := v\cdot 1\in \SRind$.
By replacing $u_i$ with $u_i - u_v$ (which has the same image in $\SRind\otimes_{\Sym(V)}k$), 
we replace $u^\sigma$ with a sum of monomials of the form $u^\tau$,
where $\tau_i = \sigma_i-1$, $\tau_j = \sigma_j$ for all $j\in B\smallsetminus\{i\}$, and $\tau_k \leq 1$ for all $k\notin B$.
Applying this procedure recursively, we may express the image of $u^\sigma$ as a sum of square-free monomials.
\end{proof}

Given a subset $S\subset [n]$, let $\langle S \rangle$ be the set of all $i$ such that $a_i$ is contained in the $k$-linear
span of $\{a_j\mid j\in S\}$.  We always have $S\subset \langle S \rangle$; if $\langle S \rangle = S$, then $S$ is called a {\bf flat}.
Given any flat $F$, let $V_F$ be the quotient of $V$ by the elements that vanish on $a_i$ for all $i\in F$.
Then we can regard $\{a_i\mid i\in F\}$ as a set of linear functionals on $V_F$ that span $V_F^*$, which allows
us to define the $\Sym(V_F)$-algebras $(\SRbc)_F$ and $\OT_F$.
When $F= [n]$, we have $V_F = V$, $(\SRbc)_F = \SRbc$, and $\OT_F = \OT$.
  
We have canonical maps
$$\mu_F:\SRbc\to(\SRbc)_F\and \nu_F:\OT\to\OT_F$$ given by setting the variables not in $F$ to zero,
as well as sections $$\a_F:(\SRbc)_F\to\SRbc \and \b_F:\OT_F\to\OT$$ taking $u_i$ to $u_i$ for all $i\in F$.
The following result is proved in \cite[3.12]{TP08}.

\begin{theorem}
We may choose a $\Sym(V)$-module isomorphism $\varphi:\SRbc\to\OT$ and a
$\Sym(V_F)$-module isomorphism $\varphi_F:(\SRbc)_F\to\OT_F$ for every flat $F$
such that the diagram
\[\tikz[->, thick]{
\matrix[row sep=10mm,column sep=10mm,ampersand replacement=\&]{
\node (a) {$\SRbc$}; \& \node (c) {$\OT$};\\
\node (b) {$(\SRbc)_F$}; \& \node (d) {$\OT_F$};\\
};
\draw (a) --node[above]{$\varphi$} (c);
\draw (a) --node[left]{$\mu_F$} (b);
\draw (b) --node[above]{$\varphi_F$}(d);
\draw (c) --node[right]{$\nu_F$} (d);
}\]
commutes.  Furthermore, these choices are unique if we require that $\varphi(1) = 1$. 
\end{theorem}

\begin{lemma}\label{the point}
Let $S\subset [n]$ be independent.
There exist constants $c_{S'}\in k$
for each independent set $S'\subset [n]$ with $\langle S'\rangle = \langle S\rangle$
such that
$$\varphi(u_S) = \sum_{S'} c_{S'} u_{S'} \in \OT.$$
\end{lemma}

\begin{proof}
Start by choosing any constants $c_\sigma$ such that $\varphi(u_S) = \sum_\sigma c_\sigma u^\sigma \in \OT$,
where the sum runs over $\sigma\in\mathbb{N}^n$.
We will show that we can kill all those $c_\sigma$ with $S\not\subset\langle\supp(\sigma)\rangle$ 
without changing the class that it represents in $\OT$.  Since $S$ is independent, this will imply that $u^\sigma = u_{S'}$
for some independent set $S'$ with $\langle S'\rangle = \langle S \rangle$.

Suppose that $F$ is a flat that does not contain $S$.  Then $\mu_F(u_S) = 0$,
so $$\nu_F\circ\varphi(u_S) = \varphi_F\circ\mu_F(u_S) = \varphi_F(0) = 0.$$
This means that $$\sum_{\langle \sigma \rangle\subset F} c_\sigma u^\sigma = 0 \in \OT_F.$$
Applying $\b_F$, this implies that $$\sum_{\langle \sigma \rangle\subset F} c_\sigma u^\sigma = 0 \in \OT.$$
Hence we may assume that $c_\sigma = 0$ for all $\sigma$ such that $\langle \supp(\sigma) \rangle\subset F$.
Since we chose $F$ to be an arbitrary flat not containing $S$, 
this means that we can assume $c_\sigma = 0$ for all $\sigma$ such that $S\not\subset \langle\supp(\sigma)\rangle$.
\end{proof}

\begin{proofmonspan}
By Lemma \ref{independence}, it is sufficient to show that, for all independent $S\subset [n]$, 
$\phi(u_S)$ may be expressed as a linear combination of elements of the form $u_{S'}$ where $S'\subset [n]$ is independent.
This is exactly the content of Lemma \ref{the point}.
\end{proofmonspan}

\subsection{A flat deformation (in the unimodular case)}\label{sec:def}
As in Section \ref{sec:defs}, fix a finite-rank lattice $N$, 
along with a list of (not necessarily distinct) nonzero primitive vectors $a_1,\ldots,a_n\in N$ that span $N$.
Let $V = N\otimes_\Z k$, and consider the associated Orlik-Terao algebra $\OT$.
Again as in Section \ref{sec:defs}, we assume that our collection of vectors is unimodular.
This implies that, for any circuit $C$, there is a decomposition 
$C = C^+\sqcup C^-$ (unique up to swapping $C^+$ and $C^-$) such that $\sum_{i\in C^+}a_i - \sum_{j\in C^-}a_j=0$.
In other words, we may always take the constants $\eta_i$ from the previous section to be $\pm 1$.
We define a {\bf signed circuit} to be a circuit equipped with a choice of decomposition.

\begin{remark}
In Section \ref{sec:defs}, we chose a simple affine hyperplane arrangement, and used this to choose a distinguished
signed circuit for each circuit.  Here we have no such affine arrangement, and there is no distinuguished choice.
\end{remark}

For each signed circuit $C$, let 
$$\fch:= \hbar^{-1}\!\left(\prod_{i\in C^+} u_i\cdot \prod_{j\in C^-} (u_j-\hbar) \;\;-\;\;
\prod_{i\in C^+} (u_i - \hbar)\cdot \prod_{j\in C^-} u_j\right),$$
and consider the $k[\hbar]$-algebra
$$\OTh := k[u_1,\ldots,u_n,\hbar]/\Ih,$$ where 
$$\Ih = \langle \fch\mid \text{$C$ a signed circuit}\rangle.$$
For any $t\in k$, let $I_t\subset k[u_1,\ldots,u_n]$ be the ideal obtained from $\Ih$ by setting $\hbar$ equal to $t$,
and let $$\OT_t := k[u_1,\ldots,u_n]/I_t.$$
In particular, we have $I_0= I$, and therefore $\OT_0$ is equal to the Orlik-Terao algebra $\OT$.

\begin{remark}\label{same OT} 
If $C$ and $C'$ are opposite signed circuits, then $\fch = -f_{C'}$, thus it is enough to pick one signed circuit
for each circuit in the definition of $\Ih$.
In particular, this means that $\Ih$ coincides with the ideal $J_1'$ defined in Corollary \ref{RT} when $k=\C$,
and therefore $\OTh$ coincides with $R'_\bT(\tX)$.  Furthermore, the $\Sym(V)[\hbar]$-algebra structure on $\OTh$
coincides with the $A_\C[\hbar]$-algebra structure on $R'_\bT(\tX)$.
\end{remark}

\begin{theorem}\label{thm-apx}
The algebra $\OTh$ is a free module over $k[\hbar]$, and is thus a flat deformation of $\OT_0$.
\end{theorem}

\begin{proof}
Let $$\Ih' := \{f\mid \hbar^kf\in\Ih\; \text{for some $k\in\N$}\}\and \OTh' := k[u_1,\ldots,u_n,\hbar]/\Ih';$$
then $\OTh'$ is a flat deformation of $\OT'_0$.
It is clear that $\OT_1 = \OT'_1$ and that we have a surjection $\OT_0\twoheadrightarrow\OT'_0$,
and therefore a closed inclusion $\Spec\OT'_0\subset\Spec\OT_0$.
Theorem \ref{thm-apx} is equivalent to the statement that this inclusion is an isomorphism.  Since $\Spec\OT_0$
is reduced and irreducible of dimension $\rk N$, it is sufficient to show that $\dim\Spec\OT'_0 = \rk N$.
Since $\OTh'$ is flat, this is equivalent to showing that $\dim\Spec\OT_1 = \rk N$.
Since we already know one inequality, we need only show that $\dim\Spec\OT_1 \geq \rk N$.

Consider the ideal $\tilde{I}_1\subset k[u_1,v_1,\ldots,u_n,v_n]$ generated by elements of the form
$$\prod_{i\in C^+} u_i\cdot \prod_{j\in C^-} v_j \;\;-\;\;
\prod_{i\in C^+} v_i\cdot \prod_{j\in C^-} u_j$$
for each signed circuit $C$, and let $\widetilde\OT_1 := k[u_1,v_1,\ldots,u_n,v_n]/\tilde{I}_1$.
Since $\Spec\OT_1$ is cut out of $\Spec\widetilde\OT_1$ by the $n$ equations $v_i = u_i - 1$
and intersects the regular locus of $\Spec\widetilde\OT_1$ nontrivially, 
it is sufficient to show that $\dim\Spec\widetilde\OT_1\geq \rk N + n$.

Consider the lattice $L := N\oplus\Z^n$
and the elements $r_i = (a_i,e_i) \in L$ and $s_i = (-a_i,e_i) \in L$.
Define a map from $\widetilde\OT_1$ to $k\{q^\ell\mid\ell\in L\}$ by sending $u_i$ to $q^{r_i}$ and $v_i$ to $q^{s_i}$.  
Since $\{r_1,s_1,\ldots,r_n,s_n\}$ spans a finite index sublattice of $L$, the induced map from the torus
$T_L := \Spec k\{q^\ell\mid\ell\in L\}$ to $\Spec\widetilde\OT_1$ is finite-to-one.  Since $\dim T_L = \rk N + n$,
this completes the proof.
\end{proof}

\begin{remark}
The quotient $$AOT_0 := k[u_1,\ldots,u_n]\;\Big{/}\;I+\langle u_1^2,\ldots,u_n^2\rangle$$
is called the {\bf Artinian Orlik-Terao algebra}.
Moseley \cite[4.5]{Moseley} studied the ring
$$AOT_\hbar := k[u_1,\ldots,u_n,\hbar]\;\Big{/}\;\Ih+\langle u_1(u_1-\hbar),\ldots,u_n(u_n-\hbar)\rangle,$$
and showed that $AOT_\hbar$ is a flat deformation of $AOT_0$ into the Varchenko-Gelfand algebra.
This is the Artinian analogue of Theorem \ref{thm-apx}, but neither result follows from the other.
\end{remark}

\bibliography{./symplectic}
\bibliographystyle{amsalpha}

\end{document}